\documentclass{amsart}
\usepackage{amsthm}
\usepackage{amsmath}
\usepackage{amsfonts}
\usepackage{amssymb}
\usepackage{mathtools}
\usepackage{bm}
\usepackage[section]{algorithm}
\usepackage{algpseudocodex}
\usepackage{tikz}
\usepackage{graphicx}
\usepackage[luatex]{hyperref}
\usepackage{color}

\usepackage{fancyhdr}
\theoremstyle{plain}
\newtheorem{thm}{Theorem}[section]
\newtheorem{lem}[thm]{Lemma}
\newtheorem{prop}[thm]{Proposition}

\newtheorem{ques}[thm]{Question}

\theoremstyle{definition}

\newtheorem{exam}[thm]{Example}

\theoremstyle{remark}
\newtheorem{rem}[thm]{Remark}

\newcommand{\intnonneg}{\mathbb{Z}_{\geq 0}}
\newcommand{\realnonneg}{\mathbb{R}_{\geq 0}}

\newcommand{\kk}{\Bbbk}
\newcommand{\Fc}{\mathcal{F}}
\newcommand{\QQ}{\mathbb{Q}}
\newcommand{\RR}{\mathbb{R}}

\newcommand{\xb}{{\bf x}}
\newcommand{\yb}{{\bf y}}
\newcommand{\ub}{{\bf u}}
\newcommand{\wb}{{\bf w}}
\newcommand{\zb}{{\bf z}}
\newcommand{\vb}{{\bf v}}

\DeclareMathOperator{\supp}{supp}
\DeclareMathOperator{\ini}{in}

\title{Non-finitely generated monoids corresponding to finitely generated homogeneous subalgebras}
\date{}
\author[A. Higashitani]{Akihiro Higashitani}
\address{Department of Pure and Applied Mathematics, Graduate School of Information
Science and Technology, Osaka University, Osaka, Japan}
\email{higashitani@ist.osaka-u.ac.jp}

\author[K. Tani]{Koichiro Tani}
\address{Department of Pure and Applied Mathematics, Graduate School of Information
Science and Technology, Osaka University, Osaka, Japan}
\email{tani-k@ist.osaka-u.ac.jp}
\subjclass{Primary: 13E15, Secondary: 13P10, 20M25}
\keywords{SAGBI basis, monoid algebra, finitely generated}

\begin{document}

\begin{abstract}
The goal of this paper is to study the possible monoids appearing 
as the associated monoids of the initial algebra of a finitely generated homogeneous $\kk$-subalgebra of a polynomial ring $\kk[x_1,\ldots,x_n]$. 
Clearly, any affine monoid can be realized since the initial algebra of the affine monoid $\kk$-algebra is itself.  
On the other hand, the initial algebra of a finitely generated homogeneous $\kk$-algebra is not necessarily finitely generated. 
In this paper, we provide a new family of non-finitely generated monoids which can be realized as the initial algebras of finitely generated homogeneous $\kk$-algebras. 
Moreover, we also provide an example of a non-finitely generated monoid which cannot be realized as the initial algebra of any finitely generated homogeneous $\kk$-algebra. 
\end{abstract}

\maketitle

\section{Introduction}
Let $\kk$ be a field, $S = \kk[x_1, \ldots, x_n]$ the polynomial ring in $n$ variables over $\kk$, and $\preceq$ a monomial order on $(\intnonneg)^n$. 
We use an abbreviation of monomials $x_1^{u_1} \cdots x_n^{u_n}$ with $x^\ub$ for $\ub = (u_1, u_2, \ldots, u_n) \in (\intnonneg)^n$. 
Given a non-zero polynomial $f = \sum c_{\ub} x^\ub \in S$ with $c_{\ub} \in \kk$, 
we define $\supp{f} \coloneqq \{\ub \in (\intnonneg)^n \mid c_{\ub} \ne 0\}$, 
$\deg_{\preceq}{f} \coloneqq \max_{\preceq}(\supp{f})$, and $\ini_{\preceq}{f} = c_{\deg_{\preceq}{f}}x^{\deg_{\preceq}{f}}$.
Let $R$ be a finitely generated $\kk$-subalgebra of $S$. We define $\deg_{\preceq}{R} \coloneqq \{\deg_{\preceq}{f} \mid f \in R \backslash \{0\}\}$ 
and $\ini_{\preceq}{R}$ the $\kk$-vector space spanned by $\{\ini_{\preceq}{f} \mid f \in R\}$, called \textit{initial algebra} of $R$. 
A subset $\Fc$ of $R$ is said to be \textit{SAGBI basis} of $R$ if the $\kk$-algebra generated by $\{\ini_{\preceq}{f} \mid f \in \mathcal{F}\}$ is equal to $\ini_{\preceq}{R}$. 
The word "SAGBI" is introduced by Robbiano and Sweedler~\cite{robbianosubalgebra} and  stands for "Subalgebra Analog to Gr\"{o}bner Bases for Ideal". 
Remark that $\ini_{\preceq}{R}$ is not necessarily a finitely generated $\kk$-algebra even if $R$ is finitely generated, 
so $R$ may have no finite SAGBI basis with respect to any monomial orders. 
The sufficient condition for $R$ to have infinite SAGBI basis is not found yet as far as the authors know. 
In \cite{kurodanewclass}, it is claimed that it will be helpful for this problem to study the mechanism of non-finite generation of the initial algebras. 
The goal of this paper is to study non-finitely generated monoids appearing as non-finitely generated homogeneous monoid $\kk$-algebras. 

Let $M$ be a monoid, a set with an operation $M \times M \rightarrow M$ that is associative and has the identity. 
For a monoid $M \subset (\intnonneg)^n$ and a field $\kk$, $\kk[M]$ is a $\kk$-vector space with the base $\{x^\ub \mid \ub \in M\}$. 
Since both $\kk[M]$ and $\ini_{\preceq}{R}$ are $\kk$-subalgebras of $S$ generated by monomials, there exists a monoid $M$ such that
\[\ini_{\preceq}{R} = \kk[M]\]
for any initial algebras $\ini_{\preceq}{R}$. Concretely, such a monoid $M$ is $\deg_{\preceq}{R}$. 
Therefore, to check whether $\ini_{\preceq}{R}$ is finitely generated or not is equivalent to check whether $M(=\deg_{\preceq}{R})$ is so. 
In that sense, the following question naturally arises. 
\begin{ques}\label{q:main}
Can we characterize non-finitely generated monoids arising from some finitely generated $\kk$-subalgebras? 
\end{ques}

Towards the solution of Question~\ref{q:main}, in this paper, we concentrate on our discussion in the case of homogeneous subalgebras of $\kk[x, y]$. 
In particular, we mainly study subalgebras generated by one homogeneous binomial $x^{\vb_1}+x^{\vb_2}$ and 
finitely many monomials $x^{\ub_1}, x^{\ub_2}, \ldots, x^{\ub_t}$. 
  
There are two main results in this paper. 
The first main result is to provide a class of non-finitely generated monoids that correspond to some finitely generated subalgebras. 
\theoremstyle{plain}
\newtheorem*{MainTheorem1}{\rm\bf Theorem~\ref{thm:fg subalg}}
\begin{MainTheorem1}
Let $\vb_1,\vb_2 \in (\intnonneg)^2$ be linearly independent over $\QQ$. 
Let $\preceq$ be a monomial order with $\vb_1 \succeq \vb_2$ and let $C \subset \realnonneg$ be the cone generated by $\vb_1, \vb_2$. 
We take $\ub_1, \ub_2, \ldots, \ub_s$ from $(\intnonneg)^2 \cap C^{\circ}$, where $C^\circ$ denotes the interior of $C$.  
Let $N$ be the monoid generated by $\vb_2$, and let $L$ be an $N$-module generated by $\ub_1, \ub_2, \ldots, \ub_s$. 
We define $M$ by setting the monoid generated by $\{\vb_1\} \cup L$. 
If $R$ is a $\kk$-algebra generated by 
\[G \coloneqq \{x^{\vb_1}+ x^{\vb_2}\} \cup \{x^\ub \mid \ub \in L\},\]
then $R$ is finitely generated. Moreover, for a monomial order $\preceq$ with $\vb_1 \succeq \vb_2$, we have $\ini_{\preceq}{R} = \kk[M]$. 
In particular, $G$ is an infinite SAGBI basis of $R$. 
\end{MainTheorem1}
  
The following second main result is to show that 
submonoids of $(\intnonneg)^2$ do not necessarily correspond to some initial algebras of finitely generated homogeneous subalgebras.
\newtheorem*{MainTheorem2}{\rm\bf Theorem~\ref{thm:nonexistence of the monoid}}
\begin{MainTheorem2}
Let $M$ be a submonoid of $(\intnonneg)^2$ generated by infinitely many irreducible elements $\{(1, n^2) \mid n \in \intnonneg\}$. 
Then, for any subalgebra $R$ generated by finitely many homogeneous polynomials in $\kk[x, y]$ and any monomial order $\preceq$ of $(\intnonneg)^2$, 
$\ini_{\preceq}{R}$ is never equal to $\kk[M]$. 
\end{MainTheorem2}

This paper is organized as follows. In Section~\ref{SAGBI basis criterion, monoids and cones}, 
we prepare the fundamental materials on SAGBI basis, monoids and cones. 
In Section~\ref{Examples of monoids and subalgebras}, we enumerate examples of monoids and subalgebras, and show the first main result as a generalization of them. 
In Section~\ref{Nonexistence of the monoid algebra as a initial algebra}, we give a proof of the other main result, Theorem~\ref{thm:nonexistence of the monoid}. 
In Section~\ref{Other examples}, we display examples that do not suit the class in Section~\ref{Examples of monoids and subalgebras}. 

\section*{Acknowledgements}
The first named author is partially supported by KAKENHI 21KK0043. 
\bigskip

\section{SAGBI basis criterion, monoids and cones}\label{SAGBI basis criterion, monoids and cones}
In this section, we introduce the fundamental materials on SAGBI basis, monoids, and cones.

We first provide the SAGBI basis criterion. 
We use the notation used in~\cite[Chapter 11]{sturmfelsgrobner} on SAGBI basis.
Algorithm~\ref{algorithm:subduction} is a modification of \cite[Algorithm 11.1]{sturmfelsgrobner}.

\begin{algorithm}[ht]
  \caption{(The subduction algorithm)}\label{algorithm:subduction}
  \begin{algorithmic}
    \Require $\mathcal{F} = \{f_1, f_2, \ldots, f_s\} \subset S$, $f \in S$
    \Ensure $q \in \kk[\mathcal{F}], r \in S$ such that $f = q + r$
    \State $q \coloneqq 0; r \coloneqq 0$
    \State $p \coloneqq f$
    \While{$p \notin \kk$}
      \State find $i_1, i_2, \ldots, i_s \in \intnonneg$ and $c \in \kk\backslash\{0\}$ such that
      \begin{align}\label{eqn:initial representation}
	 \tag{$\ast$}
       \ini_{\preceq}{p} = c \cdot {\ini_{\preceq}{f_1}}^{i_1} \cdot {\ini_{\preceq}{f_2}}^{i_2} \cdots {\ini_{\preceq}{f_s}}^{i_s}.
     \end{align}
      \If {representation~(\ref{eqn:initial representation}) exists}
      \State $q \coloneqq q + c \cdot f_1^{i_1} \cdot f_2^{i_2} \cdots f_s^{i_s}$
      \State $p \coloneqq p - c \cdot f_1^{i_1} \cdot f_2^{i_2} \cdots f_s^{i_s}$
      \Else
      \State $r \coloneqq r + \ini_{\preceq}{p}$
      \State $p \coloneqq p - \ini_{\preceq}{p}$
      \EndIf
    \EndWhile
    \State \Return $q, r$
  \end{algorithmic}
\end{algorithm}

Let $\ub_1, \ub_2, \ldots, \ub_s \in (\intnonneg)^n$ with $\ini_{\preceq}{f_i} = x^{\ub_i}$, 
$\mathcal{A} = (\ub_1, \ub_2, \ldots, \ub_s)$ the $n \times s$-matrix whose columns are $\ub_i$'s, 
and let $I_{\mathcal{A}}$ be the toric ideal of $\mathcal{A}$, i.e. the kernel of a $\kk$-algebra homomorphism
\begin{equation*}
  \kk[X_1, X_2, \ldots, X_s] \rightarrow \kk[x_1, x_2, \ldots, x_n], \quad X_i \mapsto x^{\ub_i}. 
\end{equation*}
Thanks to Proposition~\ref{prop:sagbi criterion}, we can determine whether $\mathcal{F}$ is a SAGBI basis. 

\begin{prop}[{\cite[Corollary 11.5]{sturmfelsgrobner}}]\label{prop:sagbi criterion}
Let $\{p_1, p_2, \ldots, p_t\}$ be generators of the toric ideal $I_{\mathcal{A}}$. 
Then $\mathcal{F}$ is a SAGBI basis if and only if Algorithm~\ref{algorithm:subduction} subduces 
$p_i(f_1, f_2, \ldots, f_s)$ to an element of $\kk$ for all $i$. 
\end{prop}

We use the notation used in~\cite[Chapters 1 and 2]{bruns2009polytopes}.
In this paper, let $M$ be a submonoid of $(\intnonneg)^n$. 
For $\xb \in M$, we call $\xb$ \textit{irreducible} on $M$ if there are $\yb, \zb \in M$ with $\xb = \yb + \zb$, then either $\yb$ or $\zb$ must be $\bm{0}$. 
A monoid $M$ is non-finitely generated if and only if $M$ has infinitely many irreducible elements on $M$. 
A set $N$ with an operation $M \times N \rightarrow N$ is called an $M$-\textit{module} if 
\begin{equation*}
  (\ub + \vb) + \xb = \ub + (\vb + \xb) \quad \text{and} \quad \bm{0} + \xb = \xb \quad \;\text{ for any } \ub, \vb \in M \text{ and } \xb \in N.
\end{equation*}

It is convenient to observe $\realnonneg M \coloneqq \{\sum_{i=1}^{s}a_i\xb_i \mid \xb_i \in M, a_i \in \realnonneg, s \in \intnonneg\}$ 
for determining if $M$ is finitely generated. 
For $i=1, 2, \ldots, t$, let $\sigma_i$ be a linear form on $\mathbb{R}^n$ and let $H_i$, $H_i^{+}$ be linear hyperplanes and linear closed halfspaces such that
\begin{equation*}
  H_i \coloneqq \{\xb \in \mathbb{R}^n \mid \sigma_{i}(\xb) = 0\} \;\text{ and }\; H_i^{+} \coloneqq \{\xb \in \mathbb{R}^n \mid \sigma_{i}(\xb) \geq 0\}, 
\end{equation*}
respectively. 
Note that we also often use $\wb \in \RR^n$ to describe each linear form $\sigma$, i.e., $\sigma(\xb)=\langle \xb,\wb \rangle$, 
where $\langle \cdot, \cdot \rangle$ denotes the usual inner product of $\RR^n$. 
A \textit{polyhedral cone} is defined to be an intersection of finitely many linear closed halfspaces, i.e., $C$ is written as $C=\bigcap_{i=1}^tH_i^{+}$. 
Moreover, by using Proposition~\ref{prop:monoid and cone}, we can determine if a given monoid $M$ is finitely generated.
\begin{prop}[{\cite[Corollary 2.10]{bruns2009polytopes}}]\label{prop:monoid and cone}
  A monoid $M$ is finitely generated if and only if $\realnonneg M$ is a polyhedral cone.
\end{prop}

A \textit{face} $F$ of $C$ is a non-empty intersection of a linear hyperplane $H = \{\xb \in \mathbb{R}^n \mid \sigma(\xb) = 0\}$ and $C$ satisfying $C \subset H^+$. 
Namely,  
\begin{equation*}
  F \coloneqq H \cap C = \{\xb \in C \mid \sigma(\xb) = 0\} \ne \emptyset. 
\end{equation*}
For a polyhedral cone $C = \bigcap_{i=1}^{t}H_i^{+}$, we define 
\begin{equation*}
  C^{\circ} \coloneqq \{\xb \in C \mid \sigma_{i}(\xb) > 0 \text{ for each }i\}.
\end{equation*}
Note that for all $\xb \in C$ and $\xb' \in C^{\circ}$, we have $\xb + \xb' \in C^{\circ}$. 

\bigskip

\section{Examples of monoids and subalgebras}\label{Examples of monoids and subalgebras} 
In this section, various non-finitely generated monoids are generalized (Lemma~\ref{lem:def monoid}) 
and we construct finitely generated subalgebras that correspond to monoids (Theorem~\ref{thm:fg subalg}). 

We found these examples through computational experiments by using the package 
\href{https://macaulay2.com/doc/Macaulay2/share/doc/Macaulay2/SubalgebraBases/html/index.html}{\tt{"SubalgebraBases"}}~\cite{SubalgebraBasesSource}
on \href{https://macaulay2.com/}{\tt{Macaulay2}}~\cite{M2}.  
On our experiments, we focused on subalgebras generated by one homogeneous binomial $x^{\vb_1}+x^{\vb_2}$ and 
$t$ monomials $x^{\ub_1}, x^{\ub_2}, \ldots, x^{\ub_t}$. 

First, we discuss the case $t=1$. 
\begin{prop}\label{prop:bino and mono}
Let $\vb_1,\vb_2 \in (\intnonneg)^2$ be linearly independent. 
For any $\ub \in (\intnonneg)^2$, the $\kk$-subalgebra $R = \kk[x^{\vb_1}+x^{\vb_2}, x^\ub]$ has a finite SAGBI basis. 
\end{prop}
\begin{proof}
We may assume $\vb_1 \succeq \vb_2$ without loss of generality. 

Let $\vb_1, \ub$ be linearly independent over $\mathbb{Q}$. We consider a linear relation 
\begin{equation*}
    a_1\vb_1 + a_2\ub = b_1\vb_1 + b_2\ub, 
\end{equation*}
where $a_1, a_2, b_1, b_2 \in \intnonneg$. Since $\vb_1, \ub$ are linearly independent, we have $a_1 = b_1$ and $a_2 = b_2$. 
Thus, any cancellation of initial terms in $R$ cannot occur and we obtain that $\{x^{\vb_1}+x^{\vb_2}, x^\ub\}$ is a SAGBI basis of $R$. 

Let $m\vb_1 = \ell \ub$ with some positive integers $m, \ell$. Then, we can obtain a polynomial in $R$ as follows: 
\begin{align*}
f &\coloneqq \frac{1}{\binom{m}{1}}((x^{\vb_1} + x^{\vb_2})^{m} - (x^\ub)^\ell) \\
  &= x^{(m-1)\vb_1+\vb_2} + \frac{\binom{m}{2}}{\binom{m}{1}}x^{(m-2)\vb_1+2\vb_2} + \cdots + \frac{1}{\binom{m}{1}}x^{m\vb_2}.
\end{align*}
Now, we prove $\{f, x^\ub\}$ is a SAGBI basis of $R$. Similarly to the previous case, we consider the equality 
\begin{equation*}
    a_1((m-1)\vb_1 + \vb_2) + a_2\ub = b_1((m-1)\vb_1 + \vb_2) + b_2\ub. 
\end{equation*}
Since $\vb_1$ and $\ub$ are linearly dependent while $\vb_1$ and $\vb_2$ are linearly independent, we obtain $a_1 = b_1$. 
Therefore we obtain $a_2 = b_2$. Thus, there is no relation between $(m-1)\vb_1 + \vb_2$ and $\ub$. Hence, $\{f, x^\ub\}$ becomes a SAGBI basis of $R$. 
\end{proof}

We can similarly prove the following. 
\begin{prop}\label{prop:bino and bino}
Let $\vb_1,\vb_2 \in (\intnonneg)^2$ be linearly independent and let $\ub_1,\ub_2 \in (\intnonneg)^2$ be linearly independent. 
Then the $\kk$-subalgebra $R = \kk[x^{\vb_1}+x^{\vb_2}, x^{\ub_1}+x^{\ub_2}]$ has a finite SAGBI basis. 
\end{prop}
\begin{proof}
We may assume $\vb_1 \succeq \vb_2$ and $\ub_1 \succeq \ub_2$ without loss of generality. 
If $\vb_1, \ub_1$ are linearly independent, then the proof is the same as in Proposition~\ref{prop:bino and mono}. 
  
Next, we assume that $m\vb_1 = \ell \ub_1$ with some positive integers $m, \ell$. 
If $m = \ell = 1$ and $\vb_2 = \ub_2$, then $\{x^{\vb_1}+x^{\vb_2}\}$ is a SAGBI basis of $R$. 
We set $\vb_2 \ne \ub_2$, and let 
\begin{align*}
    f &\coloneqq (x^{\vb_1}+x^{\vb_2})^m - (x^{\ub_1}+x^{\ub_2})^\ell \\
    &= \binom{m}{1}x^{(m-1)\vb_1+\vb_2} + \cdots + x^{m\vb_2} - \left(\binom{\ell}{1}x^{(\ell-1)\ub_1+\ub_2} + \cdots + x^{\ell\ub_2}\right).
  \end{align*}
  In the case $x^{(m-1)\vb_1+\vb_2} \succeq x^{(l-1)\ub_1+\ub_2}$, $\{f, x^{\ub_1}+x^{\ub_2}\}$ is a SAGBI basis of $R$.
  In the other case, $\{f, x^{\vb_1}+x^{\vb_2}\}$ is a SAGBI basis of $R$. The proof is the same as Proposition~\ref{prop:bino and mono}.
\end{proof}

We enumerate examples found in our experiments. These examples can be regarded as a generalization of~\cite[1.20]{robbianosubalgebra}. 
All monoids corresponding to these examples are non-finitely generated monoids by Lemma~\ref{lem:def monoid}.

\begin{lem}\label{lem:def monoid}
Let $\vb_1,\vb_2 \in (\intnonneg)^2$ be linearly independent and let $\preceq$ be a monomial order with $\vb_1 \succeq \vb_2$. 
Let $C$ be a cone generated by $\vb_1, \vb_2$. 
Fix $\ub_1, \ub_2, \ldots, \ub_s \in (\intnonneg)^2 \cap C^{\circ}$, let $N$ be the monoid generated by $\vb_2$, 
and let $L$ be the $N$-module generated by $\ub_1, \ub_2, \ldots, \ub_s$. We define $M$ by setting the monoid generated by $\{\vb_1\} \cup L$.  
Then $M$ is not finitely generated. 
\end{lem}
\begin{proof}
Let $\wb_1, \wb_2 \in \mathbb{R}^2$ such that 
\begin{align*}
          \langle \wb_1, \vb_1 \rangle &= 0 \text{ and } \langle \wb_1, \xb \rangle \geq 0 \quad (\forall \xb \in C);  \\
          \langle \wb_2, \vb_2 \rangle &= 0 \text{ and } \langle \wb_2, \xb \rangle \geq 0 \quad (\forall \xb \in C).
\end{align*}
In other words, $\wb_1$ and $\wb_2$ are chosen as they define the facets $\realnonneg \vb_1$ and $\realnonneg \vb_2$ of $C$, respectively. 
Since $\ub_1, \ub_2, \ldots, \ub_s \in C^{\circ}$, we have $\langle \wb_2, \ub_i \rangle > 0$ for each $i$. 
Let, say, $\ub_1$ attain $\langle \wb_2, \ub_1 \rangle \leq \langle \wb_2, \ub_i \rangle$ for each $i$. 
Let $L_1 \coloneqq \{\ub_1 + m\vb_2 \mid m \in \intnonneg\}$ be the $N$-module generated by $\ub_1$. 
Then $L_1$ is an infinite subset of $M$. 
In what follows, we prove that all elements of $L_1$ are irreducible on $M$. 
On the contrary, suppose that $$\ub_1 + m\vb_2 = \xb_1 + \xb_2$$ 
for some $m \in \intnonneg$ and $\xb_1, \xb_2 \in M \backslash \{\bm{0}\}$. 
Then $\xb_1$ and $\xb_2$ can be written as follows: 
\begin{equation}\label{eqn:sum of x_1 and x_2}
    \xb_1 = \sum a_{s} \yb_s \;\;\text{and}\;\; \xb_2 = \sum b_{t} \zb_t, 
\end{equation}
where $\yb_s, \zb_t \in \{\vb_1\} \cup L$ and $a_s, b_t \in \intnonneg$. 
Since generators of $M$ are a subset of $C \backslash \realnonneg \vb_2$, we have $\langle \wb_2, \xb \rangle > 0$ for all $\xb \in \{\vb_1\} \cup L$. 
If $\ub_i + m'\vb_2 \in L$ appears in the summand of $\xb_1$ or $\xb_2$ of \eqref{eqn:sum of x_1 and x_2}, say, in $\xb_1$, 
then we obtain that $\langle \wb_2, \xb_1 \rangle \geq \langle \wb_2, \ub_1 \rangle$. 
Since $\langle \wb_2, \xb_2 \rangle > \langle \wb_2,  \vb_2 \rangle = 0$, 
we have $\langle \wb_2, \xb_1 + \xb_2\rangle > \langle \wb_2, \ub_1 + m\vb_2\rangle$, a contradiction to $\ub_1 + m\vb_2 = \xb_1 + \xb_2$. 
Thus, no elements in $L$ appear in the summand of \eqref{eqn:sum of x_1 and x_2}. 
Hence, $\xb_1 + \xb_2$ is a positive integer multiple of $\vb_1$. We can rewrite it as like 
\begin{equation*}
    \ub_1 + m\vb_2 = \xb_1 + \xb_2 = \ell\vb_1 
\end{equation*}
with some positive integer $\ell$. However, by applying $\langle \wb_1, - \rangle$ to both sides of these equations, we obtain that 
\begin{equation*}
    0 < \langle \wb_1, \ub_1 + m\vb_2\rangle = \langle \wb_1, \ell \vb_1\rangle = 0, 
\end{equation*}
a contradiction. 

Therefore, we conclude that all elements of $L_1$ are irreducible on $M$, implying the non-finite generation of $M$. 
\end{proof}

Now, we provide a family of examples of finitely generated $\kk$-algebras whose initial algebras are equal to $\kk[M]$ 
with $M$ defined in Lemma~\ref{lem:def monoid}. This is the first main theorem of this paper. 
\begin{thm}\label{thm:fg subalg}
Work with the same notation as in Lemma~\ref{lem:def monoid}. 
Let $R$ be the $\kk$-algebra generated by $$G \coloneqq \{x^{\vb_1}+x^{\vb_2}\} \cup \{x^\ub \mid \ub \in L\}.$$
Then $R$ is finitely generated. Moreover, given a monomial order $\preceq$ with $\vb_1 \succeq \vb_2$, we have $\ini_{\preceq}{R} = \kk[M]$. 
In particular, $G$ is a SAGBI basis of some $R$ consisting of infinitely many polynomials (most of which are monomials).  
\end{thm}
\begin{proof}
First, we prove that $R$ is finitely generated. 

Since $\vb_1$ and $\vb_2$ are linearly independent and $\ub_1, \ub_2, \ldots, \ub_s \in C^{\circ}$, 
there exist positive integers $\ell_i, a_i, b_i$ such that $$\ell_i\ub_i = a_i\vb_1 + b_i\vb_2 $$ holds for each $i=1,\ldots,s$. 
Fix $i$. Let  $$f_0 = x^{\ub_i + a_i\vb_1 + b_i\vb_2} \;\; (= (x^{\ub_i})^{\ell_i+1})$$ and \begin{align*}
f_k&=x^{\ub_i+(b_i-1+k)\vb_2}\left(x^{\vb_1}+x^{\vb_2}\right)^{a_i+1-k} \\
&=\sum_{j=0}^{a_i+1-k}\binom{a_i+1-k}{j}x^{\ub_i+(a_i+1-k-j)\vb_1+(b_i-1+k+j)\vb_2} 
\end{align*}
for $k=1,\ldots,a_i$. 
Let $V$ be the $\kk$-vector space with a basis $$\{m_p\coloneqq x^{\ub_i+(a_i-p)\vb_1+(b_i+p)\vb_2} \mid p=0,1,\ldots,a_i\}.$$ 
Then $f_0, f_1, \ldots, f_{a_i}$ belong to $V$. Let 
$$A =
\begin{pmatrix}
  1 & \binom{a_i}{0} & 0 &  & 0 & 0 \\
  0 & \binom{a_i}{1} & \binom{a_i-1}{0} &  & \vdots & \vdots \\
    &  & \binom{a_i-1}{1} & \cdots & 0 &  \\
  \vdots & \vdots & \vdots &  & \binom{2}{0} & 0 \\
  & \binom{a_i}{a_i-1} & \binom{a_i-1}{a_i-2} &  & \binom{2}{1} & \binom{1}{0} \\
  0 & \binom{a_i}{a_i} & \binom{a_i-1}{a_i-1} &  & \binom{2}{2} & \binom{1}{1}
\end{pmatrix}
$$
be the $(a_i+1)\times(a_i+1)$-matrix. 
Then $$(m_0 \; m_1 \; \cdots \; m_{a_i})A=(f_0 \; f_1 \; \cdots \; f_{a_i}).$$ 
Now, we claim that $f_0,f_1,\ldots,f_{a_i}$ also form a basis of $V$ by seeing that $A$ is regular. 
We subtract $3$-rd column from $2$-nd column, $4$-th column from $3$rd column, $\ldots$, and $(a_i+1)$-th column from $a_i$-th column. 
Then $A$ is transformed into
  $$ 
  \begin{pmatrix}
    1 & \binom{a_i-1}{0} & 0 & & 0 & 0 \\
    0 & \binom{a_i-1}{1} & \binom{a_i-2}{0} & & \vdots & \vdots \\
     & & \binom{a_i-2}{1} & \cdots & 0 &  \\
     \vdots & \vdots & \vdots & & \binom{1}{0} & 0 \\
    & \binom{a_i-1}{a_i-1} & \binom{a_i-2}{a_i-2} & & \binom{1}{1} & \binom{1}{0} \\
    0 & 0 & 0 & & 0 & \binom{1}{1}
  \end{pmatrix}
  .$$
Hence, by induction on $a_i$, we can see that $A$ is regular. 

Therefore, the monomial $m_{a_i}=x^{\ub_i+(a_i+b_i)\vb_2} \in V$ can be written as a (unique) $\kk$-linear combination of $f_0, f_1, \ldots, f_{a_i}$. 
This means that the monomial $x^{\ub_i+(a_i+b_i)\vb_2}$ belongs to 
$$
  \kk[x^{\vb_1}+x^{\vb_2}, x^{\ub_i}, x^{\ub_i+b_i\vb_2}, x^{\ub_i+(b_i+1)\vb_2}, \ldots, x^{\ub_i+(a_i+b_i-1)\vb_2}]. 
$$
Similarly to the above discussion, we can show that the monomial $x^{\ub_i+(a_i+b_i+1)\vb_2}$ can be written as a $\kk$-linear combination of the polynomials 
\begin{align*}
    (x^{\ub_i})^{\ell_i}(x^{\ub_i+\vb_2}) \text{ and } x^{\ub_i+(b_i+k)\vb_2}(x^{\vb_1}+x^{\vb_2})^{a_i+1-k} \;\text{ for }k=1,\ldots,a_i. 
\end{align*}
By applying these repeatedly, we can show that all monomials $x^{\ub_i+j\vb_2}$ for $j \in \intnonneg$ belong to 
$\kk[x^{\vb_1}+x^{\vb_2}, x^{\ub_i}, x^{\ub_i+\vb_2}, \ldots, x^{\ub_i+(a_i+b_i-1)\vb_2}]$. Thus, the $\kk$-algebra generated by 
$$
  \{x^{\vb_1}+x^{\vb_2}\} \cup \bigcup_{i=1}^{s}\{x^{\ub_i}, x^{\ub_i+\vb_2}, \ldots, x^{\ub_i+(a_i+b_i-1)\vb_2}\} 
$$
coincides with $R$, which is finitely generated. 

\medskip

Next, we prove that $\ini_{\preceq}{R} = \kk[M]$. It is sufficient to prove that $G$ is a SAGBI basis of $R$. 
Namely, we prove that for any polynomial $f \in R$, 
$\ini_{\preceq}{f}$ can be written as a product of finitely many monomials in $\{\ini_{\preceq}{g} \mid g \in G\}$. 
Since $f$ can be written as
  \begin{equation}\label{eqn:sum representation of f}
    f = \sum_{i} c_i(x^{\vb_1}+x^{\vb_2})^{\alpha^{(i)}}(x^{\wb_1})^{\beta_1^{(i)}}(x^{\wb_2})^{\beta_2^{(i)}}\cdots(x^{\wb_t})^{\beta_t^{(i)}}, 
  \end{equation}
where $c_i \in \kk, \alpha^{(i)}, \beta_1^{(i)}, \beta_2^{(i)}, \ldots, \beta_t^{(i)} \in \intnonneg$ and $\wb_1, \wb_2, \ldots, \wb_t \in L$,
the initial monomial $\ini_{\preceq}{f}$ belongs to 
$$\mathcal{S}\coloneqq\supp((x^{\vb_1}+x^{\vb_2})^{\alpha^{(i)}}(x^{\wb_1})^{\beta_1^{(i)}}(x^{\wb_2})^{\beta_2^{(i)}}\cdots(x^{\wb_t})^{\beta_t^{(i)}})$$ for some $i$. 
\begin{itemize}
\item If one of $\beta_1^{(i)}, \beta_2^{(i)}, \ldots, \beta_t^{(i)}$ is not $0$, then all monomials in $\mathcal{S}$ 
can be written as $x^{m_1\ub_1+m_2\ub_2+\ldots+m_s\ub_s+a\vb_1+b\vb_2}$ with $m_1, m_2, \ldots, m_s, a, b \in \intnonneg$ 
and one of $m_1, m_2, \ldots, m_s$ is positive, say, $m_1 > 0$. Then we have 
\begin{equation*}
    x^{m_1\ub_1+m_2\ub_2+\cdots+m_s\ub_s+a\vb_1+b\vb_2} = (x^{\vb_1})^a(x^{\ub_1})^{m_1-1}(x^{\ub_2})^{m_2}\cdots(x^{\ub_s})^{m_s}(x^{\ub_1+b\vb_2}). 
\end{equation*}
\item If $\beta_1^{(i)} = \beta_2^{(i)} = \ldots = \beta_t^{(i)} = 0$, i.e., $\ini_{\preceq}{f} \in \supp{(x^{\vb_1}+x^{\vb_2})^{\alpha^{(i)}}}$, 
since $\ub_1, \ub_2, \ldots, \ub_s \in C^{\circ}$, the monomial $x^{\alpha^{(i)}\vb_1}$ cannot be written like 
$$x^{m_1\ub_1+m_2\ub_2+\cdots+m_s\ub_s+a\vb_1+b\vb_2}$$
with $m_1, m_2, \ldots, m_s, a, b \in \intnonneg$ and one of $m_1, m_2, \ldots, m_s$ is not $0$. 
Therefore, the cancellation of the monomials of the form $x^{\alpha^{(i)}\vb_1}$ in~\eqref{eqn:sum representation of f} never happens, 
i.e., $x^{\alpha^{(i)}\vb_1}$ definitely appears in $f$. 
Since $x^{\alpha^{(i)}\vb_1}$ is the strongest monomial with respect to $\preceq$ in $\supp{(x^{\vb_1}+x^{\vb_2})^{\alpha^{(i)}}}$, 
we have $\ini_{\preceq}{f} = (x^{\vb_1})^{\alpha^{(i)}}$. 
\end{itemize}

By these discussions, we see that for any $f \in R$, the initial monomial $\ini_{\preceq}{f}$ can be written 
as a product of finitely many monomials in $\{\ini_{\preceq}{g} \mid g \in G\}$. Therefore, $G$ is a SAGBI basis of $R$. 
\end{proof}

We provide three families of examples of Theorem~\ref{thm:fg subalg}. 
Each of Examples~\ref{exam:monoid of changing cone}, \ref{exam:u equal a comma b} and \ref{exam:many irreducible lines} 
generalizes the initiated example \cite[1.20]{robbianosubalgebra} of finitely generated $\kk$-algebra whose initial algebra is not finitely generated. 
In what follows, let $\preceq$ be a monomial order with $x \succeq y$. 

\begin{exam}\label{exam:monoid of changing cone}
Given $\vb_1, \vb_2 \in (\intnonneg)^2$ which are linearly independent, let $$R_1=\kk[x^{\vb_1}+x^{\vb_2}, x^{\vb_1+\vb_2}, x^{\vb_1+2\vb_2}].$$
Then \begin{equation*}
    \{x^{\vb_1}+x^{\vb_2}\} \cup \{x^{\vb_1+m\vb_2} \mid m \geq 1\} 
\end{equation*}
forms a SAGBI basis of $R_1$ with respect to $\preceq$. 
In particular, the case with $\vb_1=(1, 0), \vb_2=(0, 1)$ is the same as that of \cite[1.20]{robbianosubalgebra}. 
\end{exam}

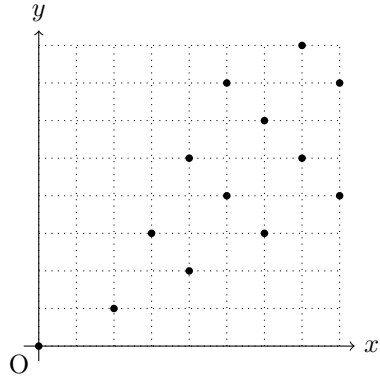
\begin{figure}[h]
  \begin{tikzpicture}[scale=0.5]
    \draw [dotted] (0,0) grid (8,8);
    \draw [->] (-0.4, 0) -- (8.4, 0) node[right]{$x$};
    \draw [->] (0, -0.4) -- (0, 8.4) node[above]{$y$};
    \fill [black] (0, 0) circle(0.1) node[below left]{O};
    \foreach \x in {1, 2, 3, 4} \fill [black] (\x*2, \x*1) circle(0.1);
    \foreach \x in {1, 2, 3} \foreach \y in {1, 2} \fill [black] (\x*2+\y*1, \x*1+\y*2) circle(0.1);
    \fill [black] (5, 7) circle(0.1);
    \fill [black] (7, 8) circle(0.1);
   \end{tikzpicture}
\caption{The monoid $M$ of Example~\ref{exam:monoid of changing cone} in the case with $\vb_1=(2, 1), \vb_2=(1, 2)$}
\end{figure}
  
\begin{exam}\label{exam:many irreducible lines}
  Given a positive integer $s$, let $$R_2=\kk[x^s+y^s, x^sy^s, x^sy^{2s}, x^{s+1}y^{s-1}, x^{s+1}y^{2s-1}, \ldots, x^{2s-1}y, x^{2s-1}y^{s+1}].$$
  Then 
  \begin{equation*}
    \{x^s + y^s\} \cup \bigcup_{i=0}^{s-1} \{x^{s+i}y^{s-i+sm} \mid m \in \intnonneg\}
  \end{equation*}
  forms a SAGBI basis of $R_2$. 
  In particular, the case $s = 1$ is the same as~\cite[1.20]{robbianosubalgebra}. 
\end{exam}

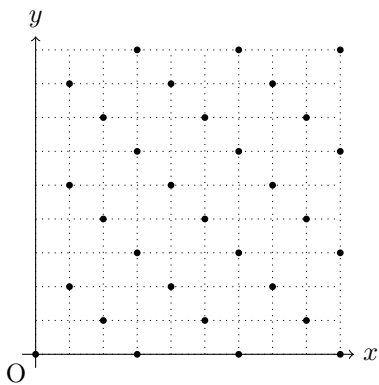
\begin{figure}[h]
    \begin{tikzpicture}[scale=0.45]
      \draw [dotted] (0,0) grid (9,9);
      \draw [->] (-0.4, 0) -- (9.4, 0) node[right]{$x$};
      \draw [->] (0, -0.4) -- (0, 9.4) node[above]{$y$};
      \fill [black] (0, 0) circle(0.1) node[below left]{O};
      \foreach \x in {1, 2, 3} \fill [black] (3*\x, 0) circle(0.1);
      \foreach \x in {1, 2, 3} \foreach \y in {1, 2, 3} \fill [black] (3*\x, 3*\y) circle(0.1);
      \foreach \x in {1, 2, 3} \foreach \y in {1, 2, 3} \fill [black] (3*\x-1, 3*\y-2) circle(0.1);
      \foreach \x in {1, 2, 3} \foreach \y in {1, 2, 3} \fill [black] (3*\x-2, 3*\y-1) circle(0.1);
      \end{tikzpicture}
      \caption{The monoid $M$ of Example~\ref{exam:many irreducible lines} in the case with $s=3$}
\end{figure}

\begin{exam}\label{exam:u equal a comma b}
  Given positive integers $a,b$, let $$R_3=\kk[x+y, x^ay^b, x^ay^{b+1}, \ldots, x^ay^{a+2b-1}].$$ Then 
  \begin{equation*}
    \{x + y\} \cup \{x^ay^m \mid m \geq b\}
  \end{equation*}
  forms a SAGBI basis of $R_3$. 
  In particular, the case $a = b = 1$ is the same as~\cite[1.20]{robbianosubalgebra}.
\end{exam}

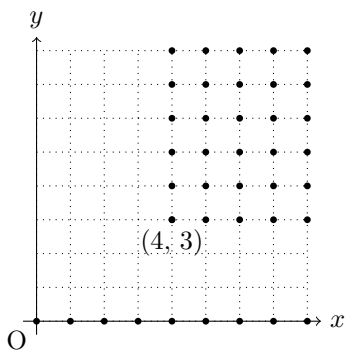
\begin{figure}[h]
    \begin{tikzpicture}[scale=0.45]
      \draw [dotted] (0,0) grid (8,8);
      \draw [->] (-0.4, 0) -- (8.4, 0) node[right]{$x$};
      \draw [->] (0, -0.4) -- (0, 8.4) node[above]{$y$};
      \fill [black] (0, 0) circle(0.1) node[below left]{O};
      \fill [black] (4, 3) circle(0.1) node[below]{(4, 3)};
      \foreach \x in {1, 2, 3, 4, 5, 6, 7, 8} \fill [black] (\x, 0) circle(0.1);
      \foreach \x in {4, 5, 6, 7, 8} \foreach \y in {3, 4, 5, 6, 7, 8} \fill [black] (\x, \y) circle(0.1);
      \end{tikzpicture}
      \caption{The monoid $M$ of Example~\ref{exam:u equal a comma b} in the case with $a=4, b=3$}
\end{figure}

\begin{rem}
Readers may have a doubt about the similarity between our examples and the ones developed in~\cite{kurodanewclass}. 
On the one hand, Examples~\ref{exam:monoid of changing cone} and \ref{exam:many irreducible lines} can be obtained via the method developed in~\cite{kurodanewclass}. 
In fact, for Example~\ref{exam:monoid of changing cone}, we may assign $x^{\vb_1}$ and $x^{\vb_2}$ instead of $x$ and $y$ in \cite[Example 2.7]{kurodanewclass}, 
and for Example~\ref{exam:many irreducible lines}, we may assign $x^s$ and $y^s$ and $U=\{x^{s+1}y^{s-1},x^{s+2}y^{s-2},\ldots,x^{2s-1}y\}$, 
where $U$ appears in (A3) of the construction developed in \cite{kurodanewclass}. 

On the other hand, most cases of Example~\ref{exam:u equal a comma b} cannot be obtained in that way. 
In fact, we can observe that the subalgebra $R$ constructed in the way of \cite{kurodanewclass} always contains $xy$ whenever $R$ contains $x+y$. 
This implies that the algebra $R_3$ cannot be equal to $R$ constructed in~\cite{kurodanewclass}. 
\end{rem}

\bigskip

\section{Nonexistence of the monoid algebra as an initial algebra}\label{Nonexistence of the monoid algebra as a initial algebra}
This section is devoted to proving the second main theorem, which provides an example of a monoid 
whose algebra cannot be realized as any initial algebra of a finitely generated homogeneous $\kk$-algebra. 
\begin{thm}\label{thm:nonexistence of the monoid}
Let $M$ be a submonoid of $(\intnonneg)^2$ generated by infinitely many irreducible elements $\{(1, n^2) \mid n \in \intnonneg\}$. 
For any subalgebra $R$ generated by finitely many homogeneous polynomials in $\kk[x, y]$ and any monomial order $\preceq$ in $(\intnonneg)^2$, 
the initial algebra $\ini_{\preceq}{R}$ cannot be equal to $\kk[M]$. 
\end{thm}
\begin{proof}
Suppose the existence of a subalgebra $R$ and a monomial order $\preceq$ such that $\ini_{\preceq}{R} = \kk[M]$. 
Let $g_n$ be a polynomial with $\ini_{\preceq}{g_n} = xy^{n^2}$ for each $n$ and let $G = \{g_0, g_1, \ldots\}$ be the reduced SAGBI basis of $R$. 
Since we assume that $R$ is generated by homogeneous polynomials of $\kk[x, y]$, the reduced SAGBI basis $G$ can be chosen as a set of homogeneous polynomials. 
Since $R$ is finitely generated and $R = \kk[g_0, g_1, \ldots]$, 
there exists $m \in \intnonneg$ such that $R = k[g_0, g_1, \ldots, g_m]$. Then $\deg_{\preceq}{R}$ is a submonoid of
the monoid generated by $\bigcup_{i=0}^{m}{\supp{g_i}}$ because all monomials appearing in polynomials of $R$ 
can be written as a product of finitely many monomials appearing in $g_0, g_1, \ldots, g_m$. 

\medskip

\noindent
{\bf (The first step)}: 
First, we prove that there exists a positive integer $a$ with $(0, a) \in \bigcup_{i=0}^{m}{\supp{g_i}}$. 
On the contrary, suppose that $(0, a) \not\in \bigcup_{i=0}^{m}{\supp{g_i}}$ for any $a \geq 1$. 
Let $C$ be the cone generated by $\bigcup_{i=0}^{m}{\supp{g_i}}$. 
Then $C$ does not include $\RR_{>0} (0,1)$. Since $C$ is finitely generated, we can take a point $(b_1, b_2) \in C$ satisfying $b_1>0$ and 
$\displaystyle y \leq \frac{b_2}{b_1}x$  for any $(x,y) \in C$. 
However, for sufficiently large $\ell$, we get $\displaystyle \ell^2 \geq \frac{b_2}{b_1}$, so $(1, \ell^2)$ is out of $C$, i.e., out of $\deg_{\preceq}{R}$, a contradiction.
    
Hence, we have $(0, a) \in \bigcup_{i=0}^{m}{\supp{g_i}}$ for some $a \geq 1$.

\medskip

\noindent
{\bf (The second step)}: Next, we prove that each $g_i$ can be written like $g_i=y^{i^2}(x+ay)$ for some $a \in \kk \setminus \{0\}$. 
Note that $a$ is independent of $i$. 

Since each $g_i$ is homogeneous, we can write 
\begin{align*}
      g_0 &= x + a_0y \\
      g_1 &= xy + \cdots + a_1y^2 + \cdots \\
      & \vdots \\
      g_m &= xy^{m^2} + \cdots + a_my^{m^2+1} + \cdots \\
      & \vdots
\end{align*}
with $a_0, a_1, \ldots, a_m, \ldots \in \kk$ and one of $a_0, a_1, \ldots, a_m$ is not $0$. 
Since $\ini_{\preceq}{g_i}=xy^{i^2}$, we get $xy^{i^2} \succeq y^{i^2+1}$. Therefore, we have $x \succeq y$.
For $b \geq 2$, monomials $x^by^{i^2+1-b}$ do not appear in $g_i$ because $x^by^{i^2+1-b} \succeq xy^{i^2}$. 
Thus we can rewrite $g_i$'s by $$g_i = xy^{i^2} + a_iy^{i^2+1}. $$ 

We prove $a_0 = a_1 = \ldots = a_m = \cdots$ by induction on $m$. In the case of $m=2$, we consider
\begin{align*}
      g &\coloneqq g_0^3g_2 - g_1^4 - (3a_0+a_2-4a_1)g_0g_1g_2 \\
      &\;= (a_{0}a_{1}-a_{0}a_{2}-2a_{1}^{2}+3a_{1}a_{2}-a_{2}^{2})x^2y^6 \\
      &\quad + (a_{0}^{3}-3a_{0}^{2}a_{1}+4a_{0}a_{1}^{2}-a_{0}a_{2}^{2}-4a_{1}^{3}+4a_{1}^{2}a_{2}-a_{1}a_{2}^{2})xy^7 \\
      &\quad + (a_{0}^{3}a_{2}-3a_{0}^{2}a_{1}a_{2}+4a_{0}a_{1}^{2}a_{2}-a_{0}a_{1}a_{2}^{2}-a_{1}^{4})y^8.
\end{align*}
Though $\ini_{\preceq}{g} \in \ini_{\preceq}{R}$, any monomials in $\{x^2y^6, xy^7, y^8\}$ cannot be written as a product of $x, xy, xy^4, xy^9,\ldots$. 
Thus, all coefficients of $g$ should be $0$. Let $h_1, h_2, h_3$ be the following polynomials in $\kk[X_0, X_1, X_2]$: 
\begin{align*}
      h_1 &= X_{0}X_{1}-X_{0}X_{2}-2X_{1}^{2}+3X_{1}X_{2}-X_{2}^{2}; \\
      h_2 &= X_{0}^{3}-3X_{0}^{2}X_{1}+4X_{0}X_{1}^{2}-X_{0}X_{2}^{2}-4X_{1}^{3}+4X_{1}^{2}X_{2}-X_{1}X_{2}^{2}; \\
      h_3 &= X_{0}^{3}X_{2}-3X_{0}^{2}X_{1}X_{2}+4X_{0}X_{1}^{2}X_{2}-X_{0}X_{1}X_{2}^{2}-X_{1}^{4}. 
\end{align*}
Then $a_0, a_1, a_2$ satisfy $h_1(a_0, a_1, a_2) = h_2(a_0, a_1, a_2) = h_3(a_0, a_1, a_2) = 0$. 
The reduced Gr\"{o}bner basis for the ideal generated by $h_1, h_2, h_3$ with respect to a lexicographic ordering is given by the following polynomials: 
\begin{align*}
      \tilde{h}_1 &= X_{1}^{4}-4X_{1}^{3}X_{2}+6X_{1}^{2}X_{2}^{2}-4X_{1}X_{2}^{3}+X_{2}^{4}=(X_1-X_2)^4; \\
      \tilde{h}_2 &= X_{0}X_{1}-X_{0}X_{2}-2X_{1}^{2}+3X_{1}X_{2}-X_{2}^{2}; \\
      \tilde{h}_3 &= X_{0}^{3}-3X_{0}^{2}X_{2}+3X_{0}X_{2}^{2}-8X_{1}^{3}+24X_{1}^{2}X_{2}-24X_{1}X_{2}^{2}+7X_{2}^{3}.
\end{align*}
Hence, $a_0, a_1, a_2$ also satisfy $\tilde{h}_1(a_0, a_1, a_2) = \tilde{h}_2(a_0, a_1, a_2) = \tilde{h}_3(a_0, a_1, a_2) = 0$.
We get $a_1 = a_2$ from $\tilde{h}_1(X_0, X_1, X_2)$, and $a_0 = a_1$ from $a_1 = a_2$ and $\tilde{h}_3(a_0, a_1, a_2) = 0$. 
Thus, we conclude $a_0 = a_1 = a_2$. 
    
Now, assume $a_0 = a_1 = \ldots = a_{i-1} = a$ for some $i \geq 3$. 
Let $b = a_i - a$ and suppose $b \ne 0$. If $i = 2k+1$ for some $k \geq 1$, then a direct computation implies that 
$$b^2g_1g_2g_{2k+1} - bg_0^{2}g_2g_{2k+1} + g_0g_1^{3}g_{2k+1} - g_0g_{k-1}g_{k+1}^{3} = b^3y^{4k^2+4k+7}(x+ay)^2$$
by using $g_j = y^{j^2}(x+ay), (0 \leq j \leq i-1)$ and $g_i = y^{i^2+1}(x+a_iy)$. The initial monomial of this polynomial is $x^2y^{4k^2+4k+7}$, 
but $x^2y^{4k^2+4k+7}$ cannot be written as a product of two monomials of $x, xy, xy^4, \ldots$ 
because any squares of integers are $0$ or $1$ modulo $4$, a contradiction.
Similarly, if $i = 2k$ for some $k \geq 2$, then we get a polynomial in $R$ 
$$b^2g_1^{2}g_{2k} - bg_0^{2}g_1g_{2k} + g_0^{4}g_{2k} - g_0g_k^{4} = b^3y^{4k^2+3}(x+ay)^2,$$
a contradiction. Thus, we have $b = 0$, i.e., $a_i = a$ for each $i \geq 0$. 

\medskip

\noindent
{\bf (The third step)}: 
Finally, we prove that $\{g_0, g_1, \ldots, g_m\}$ forms a reduced SAGBI basis of $R$. This claim contradicts the uniqueness of the reduced SAGBI basis of $R$, 
so we can prove the nonexistence of homogeneous finitely generated subalgebra $R$ with $\ini_{\preceq}{R} = \kk[M]$. 

Let $\sigma$ be an automorphism of $\kk[x, y]$ defined by 
$$ \sigma(x) = x-ay \text{ and } \sigma(y) = y. $$
Then $\sigma(g_i) = xy^{i^2}$ for each $i$. Let $I_{\mathcal{A}}$ be the toric ideal of $2 \times m$-matrix 
$$ \mathcal{A} =
    \begin{pmatrix}
      1 & 1 & \ldots & 1 \\
      0 & 1 & \ldots & m^2
    \end{pmatrix}
$$
For all binomials $t^{\alpha_0} \cdots t^{\alpha_m} - t^{\beta_0} \cdots t^{\beta_m}$ in generators of $I_{\mathcal{A}}$, 
the image of $g_0^{\alpha_0} \cdots g_m^{\alpha_m} - g_0^{\beta_0} \cdots g_m^{\beta_m}$ by $\sigma$ is
\begin{eqnarray*}
      \sigma(g_0^{\alpha_0} \cdots g_m^{\alpha_m} - g_0^{\beta_0} \cdots g_m^{\beta_m})
      = (x)^{\alpha_0} \cdots (xy^{m^2})^{\alpha_m}  - (x)^{\beta_0} \cdots (xy^{m^2})^{\beta_m} 
      = 0. 
\end{eqnarray*}
This implies that $g_0^{\alpha_0} \cdots g_m^{\alpha_m} - g_0^{\beta_0} \cdots g_m^{\beta_m} = 0$. 
Therefore, all $g_0^{\alpha_0} \cdots g_m^{\alpha_m} - g_0^{\beta_0} \cdots g_m^{\beta_m}$ subduce to an element of $\kk$. 
From Proposition~\ref{prop:sagbi criterion}, we conclude that $\{g_0, g_1, \ldots, g_m\}$ is a finite SAGBI basis, and clearly reduced. 
\end{proof}

\bigskip

\section{Other examples}\label{Other examples}
In this section, we provide examples that do not suit Section~\ref{Examples of monoids and subalgebras}. 
Since these examples were found by computer experiments, we omit proofs of non-finitely generation. 

Our first interest is whether the converse of Lemma~\ref{lem:def monoid} and Theorem~\ref{thm:fg subalg} is true.
Our question can be rewritten more precisely into the following way. 
\begin{ques}\label{ques:converse of examples}
Let $C$ be a cone generated by $\vb_1, \vb_2 \in (\intnonneg)^2$, 
and let $$\{x^{\vb_1}+x^{\vb_2}\} \cup \{x^{\ub_i}\}_{i=1}^\infty$$ be a reduced SAGBI basis of some finitely generated $\kk$-subalgebra $R \subset \kk[x, y]$, 
where each $\ub_i$ belongs to $C^{\circ}$. Then, does there exist any monoid $M$ constructed in the way as in Lemma~\ref{lem:def monoid} 
together with a monomial order $\preceq$ such that $\ini_{\preceq}{R} = \kk[M]$? 
\end{ques}

This is not true in general as Examples~\ref{exam:counterexample1} and \ref{exam:counterexample2} indicate. 
Those examples are counterexamples of Question~\ref{ques:converse of examples}. 

\begin{exam}\label{exam:counterexample1}
Let $R = \kk[x+y, x^2y, x^2y^2, x^3y^3]$ and $\preceq$ a monomial order with $x \succeq y$. 
Then a SAGBI basis of $R$ with respect to $\preceq$ seems to be 
  \begin{equation*}
    \{x+y, x^2y, x^2y^2, x^3y^3, x^2y^4, x^2y^5, \ldots\}.
  \end{equation*}
This example is similar to Example~\ref{exam:u equal a comma b} with $a=2, b=1$, but $x^2y^3$ is not contained in $R$ in this example. 
Therefore, generators of the monoid cannot be written as a union of $\{(1, 0)\}$ and any $N$-modules with $N = \intnonneg (0, 1)$. 
\end{exam}

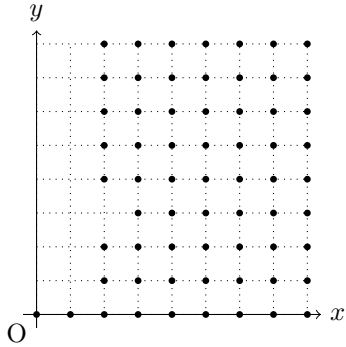
\begin{figure}[h]
    \begin{tikzpicture}[scale=0.45]
      \draw [dotted] (0,0) grid (8,8);
      \draw [->] (-0.4, 0) -- (8.4, 0) node[right]{$x$};
      \draw [->] (0, -0.4) -- (0, 8.4) node[above]{$y$};
      \fill [black] (0, 0) circle(0.1) node[below left]{O};
      \foreach \x in {1, 2, 3, 4, 5, 6, 7, 8} \fill [black] (\x, 0) circle(0.1);
      \foreach \x in {2, 3, 4, 5, 6, 7, 8} \foreach \y in {1, 2} \fill [black] (\x, \y) circle(0.1);
      \foreach \x in {3, 4, 5, 6, 7, 8} \fill [black] (\x, 3) circle(0.1);
      \foreach \x in {2, 3, 4, 5, 6, 7, 8} \foreach \y in {4, 5, 6, 7, 8} \fill [black] (\x, \y) circle(0.1);
      \end{tikzpicture}
\caption{Monoid of Example~\ref{exam:counterexample1}}
\end{figure}

\begin{exam}\label{exam:counterexample2}
Let $R = k[x^2+y^2, x^2y, x^2y^2]$ and let $\preceq$ be a monomial order with $x \succeq y$. 
Then a SAGBI basis of $R$ with respect to $\preceq$ seems to be
  \begin{equation*}
    \{x^2+y^2, x^2y, x^2y^2, x^2y^4, x^2y^6, \ldots\}.
  \end{equation*}
\end{exam}

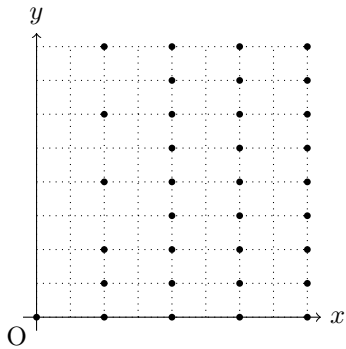
\begin{figure}[h]
    \begin{tikzpicture}[scale=0.45]
      \draw [dotted] (0,0) grid (8,8);
      \draw [->] (-0.4, 0) -- (8.4, 0) node[right]{$x$};
      \draw [->] (0, -0.4) -- (0, 8.4) node[above]{$y$};
      \fill [black] (0, 0) circle(0.1) node[below left]{O};
      \foreach \x in {1, 2, 3, 4} \foreach \y in {0, 1, 2, 4, 6, 8} \fill [black] (2*\x, \y) circle(0.1);
      \foreach \x in {2, 3, 4} \foreach \y in {3, 5, 7} \fill [black] (2*\x, \y) circle(0.1);
      \end{tikzpicture}
      \caption{Monoid of Example~\ref{exam:counterexample2}}
\end{figure}

Throughout this paper, we constructed finitely generated $\kk$-subalgebras (having an infinite SAGBI basis) 
generated by exactly one binomial $x^{\vb_1}+x^{\vb_2}$ and finitely many monomials $x^{\ub_1}, x^{\ub_2}, \ldots, x^{\ub_t}$ with
$\ub_1, \ub_2, \ldots, \ub_t \in C^\circ$, where $C=\realnonneg \vb_1 + \realnonneg \vb_2$. 
However, the condition ``$\ub_1, \ub_2, \ldots, \ub_t \in C^{\circ}$'' is not necessary for the infiniteness of SAGBI basis. 
\begin{exam}
Let $R = \kk[xy+y^2, x, xy^2]$ and $\preceq$ a monomial order with $x \succeq y$. Then a SAGBI basis of $R$ with respect to $\preceq$ seems to be 
\begin{equation*}
    \{x, xy+y^2, xy^2, 2xy^3+y^4, xy^4, 3xy^5+y^6, \ldots\}
\end{equation*}
Computing SAGBI basis of $R$, monomials and binomials appear in this SAGBI basis alternately. 
This example is almost the same as~\cite[4.11]{robbianosubalgebra} up to sign. 
\end{exam}

Finally, we introduce the most complicated example through our experiments.
\begin{exam}\label{exam:complicated example}
Let $R = \kk[x^2-y^2, x^3-y^3, x^4-y^4]$ and $\preceq$ a monomial order with $x \succeq y$. Computing a SAGBI basis of $R$, 
we observe that the following monomials appear as initial terms: 
\begin{eqnarray*}
    x^2, x^3, x^2y^2, x^3y^3, x^5y^7, x^6y^8, x^6y^{10}, x^7y^{11}, x^7y^{13}, x^8y^{14}, x^8y^{16}, x^9y^{17}, x^9y^{19}, \ldots.
\end{eqnarray*}
Thus, generators of $M$ which satisfy $\ini_{\preceq}{R} = \kk[M]$ are
\begin{equation*}
    (2, 0), (3, 0), (2, 2), (3, 3), (5, 7), (6, 8), (6, 10), (7, 11), (7, 13), (8, 14), (8, 16), (9, 17), (9, 19), \ldots.
\end{equation*}
  All points following $(3, 3)$ are contained in
\begin{equation*}
    \{(5, 7) + m(1, 3) \mid m \in \intnonneg\} \cup \{(6, 8) + m(1, 3) \mid m \in \intnonneg\}.
\end{equation*}
However, first four points $(2, 0), (3, 0), (2, 2), (3, 3)$ are not contained there, 
and $(2, 2) + (1, 3) = (3, 5)$ and $(3, 3) + (1, 3) = (4, 6)$ are not contained in $M$. 
Moreover, the monomial $xy^3$ does not appear in generators of $R$. 

In contrast, let us change the signs of generators of $R$, i.e. let $R = \kk[x^2+y^2, x^3+y^3, x^4+y^4]$. 
Then $R$ has a finite SAGBI basis with respect to a monomial order with $x \succeq y$, which is 
  \begin{equation*}
    \{x^2+y^2, x^3+y^3, x^2y^2, x^3y^3\}.
  \end{equation*}
\end{exam}

\begin{figure}[ht]
  \centering
  \includegraphics[scale=0.5]{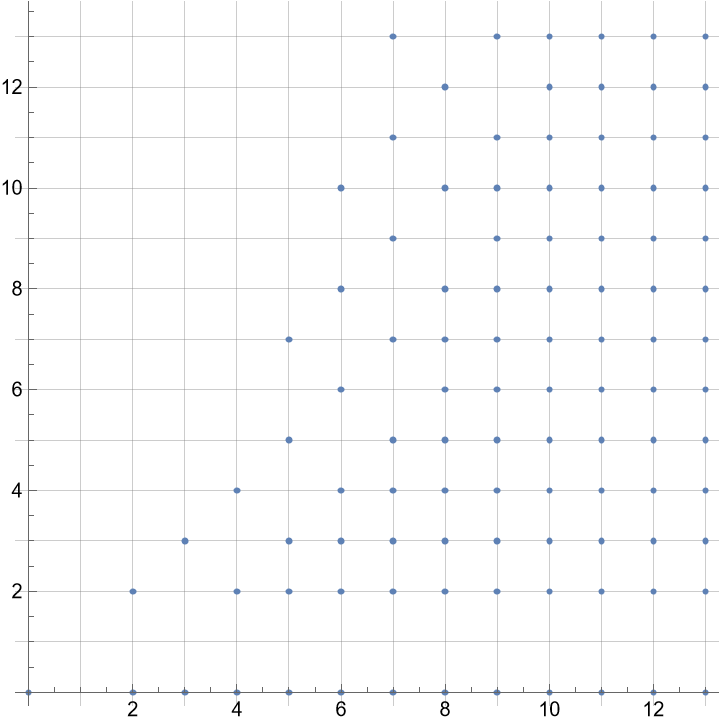}
  \caption{Non-finitely generated monoid of Example~\ref{exam:complicated example} with {\tt Mathematica}}
\end{figure}

\bibliographystyle{amsplain}
\bibliography{paperref}

\providecommand{\bysame}{\leavevmode\hbox to3em{\hrulefill}\thinspace}
\providecommand{\MR}{\relax\ifhmode\unskip\space\fi MR }
\providecommand{\MRhref}[2]{%
  \href{http://www.ams.org/mathscinet-getitem?mr=#1}{#2}
}
\providecommand{\href}[2]{#2}
\begin{thebibliography}{1}

\bibitem{bruns2009polytopes}
W.~Bruns and J.~Gubeladze, \emph{Polytopes, {R}ings, and {K}-{T}heory}, Springer Monographs in Mathematics, Springer New York, 2009.

\bibitem{SubalgebraBasesSource}
Michael Burr, Oliver Clarke, Timothy Duff, Jackson Leaman, Nathan Nichols, and Elise Walker, \emph{{SubalgebraBases: Canonical subalgebra bases (aka SAGBI/Khovanskii bases). Version~1.4}}, A \emph{Macaulay2} package available at \url{https://github.com/Macaulay2/M2/tree/master/M2/Macaulay2/packages}.

\bibitem{M2}
Daniel~R. Grayson and Michael~E. Stillman, \emph{Macaulay2, a software system for research in algebraic geometry}, Available at \url{http://www2.macaulay2.com}.

\bibitem{kurodanewclass}
Shigeru Kuroda, \emph{A new class of finitely generated polynomial subalgebras without finite sagbi bases}, Proc. Amer. Math. Soc. \textbf{151} (2023), no.~2, 533--545.

\bibitem{robbianosubalgebra}
Lorenzo Robbiano and Moss Sweedler, \emph{Subalgebra bases}, Commutative Algebra (Berlin, Heidelberg) (Winfried Bruns and Aron Simis, eds.), Springer Berlin Heidelberg, 1990, pp.~61--87.

\bibitem{sturmfelsgrobner}
B.~Sturmfels, \emph{Gr{\"o}bner bases and convex polytopes}, Univ. Lectures Series, vol.~8, American Mathematical Soc., 1996.

\end{thebibliography}
\end{document}